\documentclass{amsart}
\usepackage{amsmath, amssymb,epic,graphicx,mathrsfs,enumerate}
\usepackage[all]{xy}
\usepackage{color}
\usepackage{comment}

\usepackage{amsthm}
\usepackage{amssymb}
\usepackage{latexsym}
\usepackage{longtable}
\usepackage{epsfig}
\usepackage{amsmath}
\usepackage{hhline}

\newtheorem{thm}{Theorem}%[section]
\newtheorem{cor}[thm]{Corollary}

 \newtheorem{defn}[thm]{Definition}

\numberwithin{equation}{section}

\renewcommand{\footnote}{\endnote}
\newcommand{\ignore}[1]{}\makeglossary

\begin{document}
	\bibliographystyle{amsplain}
%	\subjclass{ 20D10, 05C25}
%	\keywords{groups generation; waiting time; Sylow subgroups; permutations groups}
	\title[A generalization of  Bernhard H. Neumann's question]{A generalization of a question\\ asked by B. H. Neumann}

	\author{Andrea Lucchini}
	\address{Andrea Lucchini\\ Universit\`a degli Studi di Padova\\  Dipartimento di Matematica \lq\lq Tullio Levi-Civita\rq\rq\\ Via Trieste 63, 35121 Padova, Italy\\email: lucchini@math.unipd.it}
%	\thanks{Partially supported by Universit\`a di Padova (Progetto di Ricerca di Ateneo: \lq\lq Invariable generation of groups\rq\rq).}

	\begin{abstract}Let $w \in F_2$ be a word and let $m$ and $n$ be two positive integers. We say that a finite group $G$ has the $w_{m,n}$-property if however a set $M$ of $m$ elements and a set $N$ of $n$ elements of the group is chosen, there exist
		at least one element of $x \in M$ and at least one element of $y \in M$ such that
		$w(x,y)=1.$	Assume that there exists a constant $\gamma < 1$ such that whenever $w$ is not an identity in a finite group $X$, then the probability that $w(x_1,x_2)=1$ in $X$ is at most $\gamma.$
		If $m\leq n$ and $G$ satisfies the $w_{m,n}$-property, then either
		$w$ is an identity in $G$ or $|G|$ is bounded in terms of $\gamma, m$ and $n$.
		We apply this result to the 2-Engel word.
\end{abstract}
	\maketitle

%%%%%%%%%%%%%%% ABOUT DIAMETER %%%%%%%%%%%%%%%%%%%%%%%%%%%%%%%%%
In 2001 Bernhard H. Neumann  asked the following question \cite{neu}: {\sl{let $G$ be a finite group  and assume
that however a set $M$ of $m$ elements and a set $N$ of $n$ elements of the group is chosen,
at least one element of $M$ commutes with at least one element of $N .$ What relations between $|G|, m, n$ guarantee that $G$ is abelian?}}

A partial answer has been given by A. Abdollahi, A. Azad, A. Mohammadi Hassanabadi
	and M. Zarrin \cite{aamz}. They proved that there exists a function $f: \mathbb N \times \mathbb N \to \mathbb N$  with the following property.
	Let $G$ be a finite group and assume
	that however a set $M$ of $m$ elements and a set $N$ of $n$ elements of the group is chosen,
	at least one element of $M$ commutes with at least one element of $N.$
	If ${|G| > f(n,m)}$, then $G$ is abelian.
	It follows from the proof that one can take $f(n,m)={c^{m+n}\max\{m,n\}}$
	where $c$ is a constant appearing in the following theorem, proved by L. Pyber in 1987 \cite{py}: {if a finite group $G$ contains at most $n$ pairwise non-commuting elements, then $|G/Z(G)|\leq  c^n$.} An alternative elementary proof, with a better estimation of the function $f$ can be obtained as a corollary of the following theorem.

\begin{thm}\label{prin}\cite[Theorem 1]{al}
	Let $\mathfrak X$ be a class of groups and suppose that there exists
	a real positive number $\gamma$ with the following property: if $X$ is a finite group and the probability that two randomly chosen elements of $X$ generate a group in  $\mathfrak X$ is
	greater than $\gamma,$ then $X$ is in  $\mathfrak X$. Assume that a finite group $G$ is such that for every two subsets $M$
	and $N$ of cardinalities $m$ and $n,$ respectively, there exist $x \in M$ and $y \in N$ such that $\langle x, y \rangle\in \mathfrak X.$ If $m\leq n,$ then
	either $G \in \mathfrak X$ or $$|G|\leq \left(\frac{2}{1-\gamma}\right)^m(n-1).$$ 
\end{thm}

\begin{cor}\label{coro}Let $G$ be a finite group  and assume
	that however a set $M$ of $m$ elements and a set $N$ of $n$ elements of the group is chosen,
	at least one element of $M$ commutes with at least one element of $N .$
	If $m\leq n$ and $G$ is not abelian, then $$|G|\leq \left(\frac{16}{3}\right)^m(n-1).$$
\end{cor}

\begin{proof}
W. H. Gustafson \cite{gu} proved that if $G$ is a finite non-abelian group, then the probability that a randomly chosen pair of elements of $G$ commutes is at most $\frac{5}{8}.$ So the statement follows immediately applying Theorem \ref {prin} to the class of finite abelian groups and taking $\gamma=\frac{5}{8}.$ 
\end{proof}

In this short note we show that, essentially with the same arguments, it can be proved a result similar to Theorem \ref{prin}, but involving words in place of classes, and in particular, a result similar to Corollary \ref{coro},  involving the 2-Engel word in place of the commutator word.

\begin{defn}Let $w \in F_2$ be a word and let $m$ and $n$ be two positive integers. We say that a finite group $G$ has the $w_{m,n}$-property if however a set $M$ of $m$ elements and a set $N$ of $n$ elements of the group is chosen, there exist
	at least one element of $x \in M$ and at least one element of $y \in M$ such that
	$w(x,y)=1.$
\end{defn}

Our main result is the following.
\begin{thm}\label{main}Assume that a word $w \in F_2$ has the property that there exists a constant $\gamma < 1$ such that whenever $w$ is not an identity in a finite group $X$, then the probability that $w(x_1,x_2)=1$ in $X$ is at most $\gamma.$
	If $m\leq n$ and a finite group $G$ satisfies the $w_{m,n}$-property, then either
	$w$ is an identity in $G$ or
$$|G|\leq \left(\frac{2}{1-\gamma}\right)^m(n-1).$$
\end{thm}

The proof of Theorem \ref{prin} relies on
the	K\"{o}v\'{a}ri-S\'{o}s-Tur\'{a}n theorem \cite{kst}. %, stating that, if $m\leq n$ are two positive integers, then a graph with $t$ vertices and at least $((n-1)^{1/m}t^{2-1/m}+(m-1)t)/2$ edges, contains a copy of the complete bipartite graph $K_{m,n}.$
 The proof of Theorem \ref{main} is quite similar, but requires a version of K\"{o}v\'{a}ri-S\'{o}s-Tur\'{a}n theorem for direct graphs (see for example \cite[Section 3]{swa}).  Let  $\vec K_{r,s}$ be the
 complete bipartite directed graph in which the vertex set is a disjoint union $A \cup B$ with $|A| = r$ and $|B| = s,$ and an arc is
 directed from each vertex of $A$ to each vertex of $B.$ 
\begin{thm}[K\"{o}v\'{a}ri-S\'{o}s-Tur\'{a}n]\label{kst}
	Let $\vec \Gamma = (V, \vec E)$ be a directed graph with $|V|= t$.
vertices. Suppose that $\vec \Gamma$ does not contain a copy of $\vec K_{r,s}.
$ Then $$|\vec E|\leq (s-1)^{\frac{1}{r}}t^{2-\frac{1}{r}}+(r-1)t.$$
\end{thm}

\begin{proof}[Proof of Theorem \ref{main}]
Suppose that $G$ satisfies the	$w_{m,n}$-property. Consider the direct graph  $\Gamma_w(G)$ whose vertices are the elements of $G$ and in which there is an edge  $x_1\mapsto x_2$ if and only if $w(x_1,x_2)\neq 1.$  If  $w$ is not an identity in $G$,
then the probability that two vertices of $\Gamma_{w}(G)$ are joined by an edge is at least $1-\gamma$, so we must have
	\begin{equation}\label{uno}
	\eta \geq {(1-\gamma)|G|^2}.
	\end{equation}
	On the other hand, since $G$ satisfies the $w_{m,n}$-property, the graph  $\Gamma_w(G)$ cannot contain the   $\vec K_{m,n}$ as a subgraph.
 By Theorem \ref{kst},
	\begin{equation}\label{due}\eta\leq {(n-1)^{1/m}|G|^{2-1/m}+(m-1)|G|}.
	\end{equation}
	Combining (\ref{uno}) and \ref{due}, we deduce
	\begin{equation}\label{tre}
	\left( \frac{n-1}{|G|}\right)^{1/m}+\frac{n-1}{|G|}\geq 
		\left( \frac{n-1}{|G|}\right)^{1/m}+\frac{m-1}{|G|}\geq 1-\gamma.
	\end{equation}
We may assume $|G|\geq n-1$. This implies 	$\left(\frac{n-1}{|G|}\right)^{1/m}\geq\frac{n-1}{|G|}$ and therefore it follows from (\ref{tre}) that
\begin{equation}
\left(\frac{n-1}{|G|}\right)^{1/m}\geq \frac{1-\gamma}2.
\end{equation}
This implies
$$|G|\leq \left(\frac{2}{1-\gamma}\right)^m(n-1).\qedhere$$
\end{proof}

\begin{cor}\label{ma}
Let $w=[x,y,y]$ be the 2-Engel word. There exists a constant $\tau$ such that if $m\leq n$ and  $G$ satisfies the $w_{m,n}$-property, then either
$w$ is an identity in $G$ or
$|G|\leq \tau^m(n-1).$
\end{cor}
\begin{proof}
By \cite{dn}, there exists a constant $\delta$ such that if $[x,y,y]$ is not the identity in $G,$ then the probability that
$[g_1,g_2,g_2]=1$ in $G$ is at most $\delta.$ By Theorem \ref{main}, we may take
$\tau=\frac{2}{1-\delta}.$
\end{proof}

\end{document}